\documentclass[a4paper,12pt]{amsart} 

\usepackage{amssymb,amsmath}
\usepackage{pb-diagram}
\usepackage{pb-xy}
\usepackage[cmtip,arrow]{xy}

\newtheorem{thm}{Theorem}[section] 
\newtheorem*{thm*}{Theorem}
\newtheorem{prop}[thm]{Proposition}
\newtheorem{cor}[thm]{Corollary} 
\newtheorem{lem}[thm]{Lemma}
  
\theoremstyle{definition} 
 
\newtheorem{rem}[thm]{Remark} 
\newtheorem{exa}[thm]{Example}

\numberwithin{equation}{section}

\newcommand{\skal}[2]{\langle #1,#2\rangle}
\newcommand{\alg}[1]{\mathfrak{#1}}

\begin{document}

\title{An integral formula for $G_2$--structures}
\author{Kamil Niedzia\l omski}
\date{}

 
\address{
Department of Mathematics and Computer Science \endgraf
University of \L\'{o}d\'{z} \endgraf
ul. Banacha 22, 90-238 \L\'{o}d\'{z} \endgraf
Poland
}
\email{kamil.niedzialomski@wmii.uni.lodz.pl}

\begin{abstract}
We apply the integral formula obtained by the author for a general $G$--structure to the case of $G=G_2$. We derive an integral formula relating curvatures and some quadratic invariants of the endomorphism induced by the intrinsic torsion. We interpret the formula in certain special cases.
\end{abstract}

\maketitle

\section{Introduction}

In this note we continue investigations initiated by the author in the article \cite{KN}. In the mentioned paper we derived an integral formula for a general Riemannian $G$--structure and applied it to almost Hermitian and almost contact metric structures. Here, we deal with a $G_2$ case. Let us briefly introduce the setting.

Assume $(M,g)$ is an oriented $n$ dimensional Riemannian manifold equipped with a $G$--structure, where $G$ is a subgroup of $SO(n)$. This is equivalent to existence of reduction of the structure group of oriented frame bundle $SO(M)$ to $G$. Among such $G$--structures we distinguish special ones called {\it parallel}. It means that the Levi-Civita $\nabla$ connection is in fact a $G$--connection. Equivalently, the connection form of $\nabla$ has values in the Lie algebra $\alg{g}$ of $G$. Here, we are interested in non--parallel $G$--structures. The defect of a structure to be parallel is measured by the $(1,2)$ tensor -- the intrinsic torsion $\xi$. In \cite{KN}, the author considered a vector field $\chi={\rm tr}\,\xi$ and computed its divergence. It depends on the intrinsic torsion itself and some "components" of the scalar curvature. Integrating over closed Riemannian manifolds and applying Stokes' theorem we get desired integral formula. Analogous results were obtained by Bor and Hernandez Lamoneda \cite{BH1, BH2}, where the authors apply Bochner--type formula for forms being stabilizers of $G$'s. The approach in \cite{KN} by the author, despite different, leads to the same formulas. However, components in these two mentioned formulas are selected differently with a different approach. This motivates the study of $G_2$--structures with the use of the result in \cite{KN}. It should be noticed that the relation between certain components of the objects defining $G_2$--structure and the scalar curvature, this time not integral correspondence, has been also considered in \cite{RB} and \cite{FI2}. More precisely, Bryant \cite{RB} uses, so called, torsion forms, whereas Ivanov and Friedrich \cite{FI2} use (skew--symmetric) torsion tensor of a $G_2$--connection. We show correlations with our results.

$G_2$--structures on $7$--dimensional Riemannian manifolds are widely studied. Its implications to physics are very important. Moreover, the application of a general integral formula obtained in \cite{KN} to the $G_2$ case seems to be interesting itself. The main result of this note reads as follows
\begin{equation*}
\frac{1}{6}\int_M s\,{\rm vol}_M=\int_M -\frac{3}{2}i_0(T)+6\sigma_2(T)\,{\rm vol}_M,
\end{equation*}
where we assume that $M$ is closed. Let us describe all objects from the above formula. On the left hand side, $s$ is the scalar curvature. On the right hand side, $i_0$ and $\sigma_2$ are two quadratic invariants of the associated with the intrinsic torsion endomorphism $T$. In fact, the space $T^{\ast}M\otimes\alg{g}_2^{\bot}(TM)$ of all possible intrinsic torsions is isomorphic as a $G_2$--module to the space ${\rm End}(TM)$ of all endomorphisms of the tangent bundle.

The article is organized as follows: in the next section we recall and state necessary algebraic fasts needed later. In particular, we study the relations between quadratic invariants of the group $G_2$. One of them, $\sigma_2$, which is in fact a $SO(7)$ invariant, is just the second elementary function associated with the matrix/endomorphism via the characteristic polynomial. $i_0$ is obtained with the use of the cross product on $\mathbb{R}^7$, which is by definition $G_2$--invariant. In the third section we recall the formula for the intrinsic torsion in the case of a $G_2$--structure. We rely on the article \cite{FMC}. In the fourth section we state the intergral formula obtained by the author in \cite{KN} for a general $G$--structures and state and prove its version for $G_2$--structures. The final section is devoted to certain examples.

\section{Algebraic background}

Consider $\mathbb{R}^7$ with the standard basis $e_1,\ldots,e_7$ and consider a $3$--form $\varphi$ given by
\begin{equation}\label{eq:varphi0}
\varphi=e^{123}+e^{145}+e^{167}+e^{246}-e^{257}-e^{347}-e^{356}.
\end{equation}
Then, a Lie group $G_2$ is a subgroup of $SO(7)$ of elements that fix $\varphi$. $G_2$ is a compact, $14$ dimensional Lie group. $G_2$ acts irreducibly on $\mathbb{R}^7$ and preserves the metric and orientation for which the basis $e_1,\ldots,e_7$ is oriented orthonormal. Denote it by $\skal{\,}{\,}=\skal{\,}{\,}_{\varphi}$ and the induced Hodge $\star$--operator by $\star=\star_{\varphi}$.

\subsection{The Lie algebra $\alg{g}_2$}
In order to describe the Lie algebra $\alg{g}_2$ of $G_2$ we follow closely \cite{RB}. Consider the $\varepsilon$--notation: these are unique numbers $\varepsilon_{ijk}$ and $\varepsilon_{ijkl}$ such that 
\begin{equation*}
\varphi=\frac{1}{6}\sum_{i,j,k}\varepsilon_{ijk}e^{ijk},\quad
\star\varphi=\frac{1}{24}\sum_{i,j,k,l}\varepsilon_{ijkl} e^{ijkl}.
\end{equation*}
These numbers satisfy \cite{RB}
\begin{align*}
\sum_{i,j}\varepsilon_{ijk}\varepsilon_{ijl}&=6\delta_{kl}, \\
\sum_{i}\varepsilon_{ijk}\varepsilon_{ipq} &=\varepsilon_{jkpq}+\delta_{jp}\delta_{kq}-\delta_{jq}\delta_{kp}.
\end{align*}
Reformulation of these relations will be given via the cross product below.

The Lie algebra $\alg{g}_2$ and its orthogonal complement $\alg{g}_2^{\bot}$ is $\alg{so}(7)$ may be described with the use of the $\varepsilon$--notation as follows: Let $a=(a_{ij})\in\alg{so}(7)$. Then, $a\in\alg{g}_2$ if and only if $\sum_{j,k}\varepsilon_{ijk}a_{jk}=0$ for all $i$, whereas, $a\in\alg{g}_2^{\bot}$ if and only if $a_{ij}=\sum_k\varepsilon_{ijk}v_k$ for some vector $v=\sum_i v_ie_i\in\mathbb{R}^7$. Denote a skew--symmetric matrix induced by a vector $v$ as above by $A_v$. From this considerations we have
\begin{equation}\label{eq:so7g2}
\alg{so}(7)=\alg{g}_2\oplus \mathbb{R}^7,
\end{equation}
where the decomposition is orthogonal. In other words, $\alg{g}_2^{\bot}=\mathbb{R}^7$. It follows that it is also $G_2$-invariant and irreducible. Moreover, consider a map $p:\alg{so}(7)\to\mathbb{R}^7$ given by
\begin{equation}\label{eq:defp}
p(a_{ij})=\sum_{i,j,k}\varepsilon_{ijk}a_{jk}e_i.
\end{equation}
Then, the kernel of $p$ is $\alg{g}_2$ and we have the following relations \cite{RB}
\begin{align*}
p(A_v) &=6v, \\
p(A_uA_v) &=3A_v(u)=-3A_u(v).
\end{align*}
It implies that $a=(a-A_{\frac{1}{6}p(a)})+A_{\frac{1}{6}p(a)}$ is a decomposition of $a\in\alg{so}(7)$ with respect to the splitting \eqref{eq:so7g2}. 

The algebraic lemma below will be very useful in the following section.
\begin{lem}\label{lem:g2ort}
Let $u,v\in\mathbb{R}^7$. Then the $\alg{g}_2^{\bot}$--component of $[A_u,A_v]$ with respect to \eqref{eq:so7g2} equals
\begin{equation}\label{eq:g2bot}
[A_u,A_v]_{\alg{g}_2^{\bot}}=A_{A_v(u)}.
\end{equation}
\end{lem}
\begin{proof}
This is an easy consequence of relations for a map $p$ and a remark preceding the Lemma. Indeed, we have
\begin{equation*}
p([A_u,A_v])=p(A_uA_v)-p(A_vA_u)=6A_v(u). \qedhere
\end{equation*}
\end{proof}

\subsection{Some representation theory}
Consider now the space $(\mathbb{R}^7)^{\ast}\otimes\mathbb{R}^7$ of all endomorphisms of $\mathbb{R}^7$. Since the inner product $\skal{\,}{\,}$ is $G_2$--invariant, it follows that this space as a $G_2$--module can be identified with a space of $2$--tensors $(\mathbb{R}^7)^{\ast}\otimes(\mathbb{R}^7)^{\ast}$. Thus, under the action of $G_2$ it splits into the following direct sum of irreducible modules \cite{BG,RB}
\begin{align*}
(\mathbb{R}^7)^{\ast}\otimes\mathbb{R}^7
&\equiv(\mathbb{R}^7)^{\ast}\otimes(\mathbb{R}^7)^{\ast}
\equiv\mathbb{R}\oplus S^2_0(\mathbb{R}^7)^{\ast}\oplus\Lambda^2_{(14)}(\mathbb{R}^7)^{\ast}\oplus\Lambda^2_{(7)}(\mathbb{R}^7)^{\ast}\\
&\equiv\mathbb{R}\oplus S^2_0(\mathbb{R}^7)^{\ast}\oplus\alg{g}_2\oplus\mathbb{R}^7,
\end{align*}
where the lower index in brackets denotes the dimension of the irreducible representation and $S^2_0$ denotes the space of symmetric trace free tensors.

\subsection{The cross product on $\mathbb{R}^7$}
Let us recall the relation between $G_2$, hence consequently $\alg{g}_2$, and a cross product on $\mathbb{R}^7$ (for details see \cite{BG, SW}). In is known \cite{BG} that there is a unique (up to orthogonal isomorphism) cross product on $\mathbb{R}^7$, i.e. a skew--symmetric map $(u,v)\mapsto u\times v$, $u,v\in\mathbb{R}^7$ such that
\begin{equation*}
\skal{u\times v}{u}=\skal{u\times v}{v}=0\quad\textrm{and}\quad |u\times v|^2=|u|^2|v|^2-\skal{u}{v}^2.
\end{equation*}
It satisfies the following rules
\begin{align}
\skal{u\times v}{w} &=\skal{u}{v\times w},\label{eq:crossprod1}\\
u\times(u\times w) &=\skal{u}{w}u-|u|^2w, \label{eq:crossprod2}\\
u\times(v\times w) &=-v\times (u\times w)+\skal{u}{w}v+\skal{v}{w}u-2\skal{u}{v}w. \label{eq:crossprod3}
\end{align} 

We may choose cross product to be compatible with the metric, i.e. the $3$--form $\varphi$ is given by $\varphi(u,v,w)=\skal{u\times v}{w}$. Thus, we may describe $G_2$ as
\begin{equation*}
G_2=\{g\in SO(7)\mid g(u\times v)=gu\times gv\}
\end{equation*}
and therefore
\begin{equation*}
\alg{g}_2=\{a\in\alg{so}(7)\mid a(u\times v)=a(u)\times v+u\times a(v)\}.
\end{equation*}
The $\varepsilon$--notation reads $e_i\times e_j=\sum_k\varepsilon_{ijk}e_k$ and $\varepsilon_{ijkl}=\skal{e_i\times e_j}{e_k\times e_l}$ for distinct $(i,j,k,l)$. This follows from the following formula for the Hodge star of $\varphi$ \cite{FG}
\begin{equation*}
\star\varphi(u,v,w,z)=\skal{u}{v\times w\times z}+\skal{u}{v}\skal{z}{w}-\skal{u}{w}\skal{z}{v}.
\end{equation*}
In particular, the skew--symmetric endomorphism $A_v$ is a cross product by $v$, i.e., $A_v(u)=u\times v$. Hence, \eqref{eq:g2bot} may be rewritten as
\begin{equation}\label{eq:g2botalternative}
[A_u,A_v]_{\alg{g}_2^{\bot}}=A_{u\times v}.
\end{equation}

\subsection{Quadratic invariants of $G_2$}
Consider now an endomorphism $\alpha$ of $\mathbb{R}^7$. By an elementary (symmetric) functions of $\alpha$ we mean the unique numbers $\sigma_i(\alpha)$ such that
\begin{equation*}
{\rm det}(\alpha-\lambda I)=\sum_i (-1)^i\sigma_{7-i}(\alpha)\lambda^i.
\end{equation*}
In particular, $\sigma_0(\alpha)=1$, $\sigma_1={\rm tr}\,\alpha$ and $\sigma_7(\alpha)=\det\alpha$. We are mainly interested in a quadratic invariant $\sigma_2(\alpha)$ (of the orthogonal group). Define the following numbers associated with an endomorphism $\alpha$:
\begin{align*}
i_0(\alpha) &=\sum_{i,j}\skal{\alpha(e_i)\times\alpha(e_j)}{e_i\times e_j}, \\
i_1(\alpha) &=\sum_{i,j}\skal{\alpha(e_i)\times e_i}{\alpha(e_j)\times e_j}, \\
i_2(\alpha) &=\sum_{i,j}\skal{\alpha(e_i)\times e_j}{\alpha(e_j)\times e_i}.
\end{align*}
Notice that $i_0,i_1,i_2$ and $\sigma_2, \sigma_1^2, |\cdot|^2$ are quadratic invariants of $G_2$, i.e., $i_j(g\alpha)=i_j(\alpha)$, $g\in G_2$, and the same holds for $\sigma_2$, $\sigma_1^2$ and $|\cdot|^2$.

\begin{lem}\label{lem:endcrossprod}
The following relations hold:
\begin{align*}
i_1(\alpha) &=-i_0(\alpha)+|\alpha|^2+4\sigma_2(\alpha)-\sigma_1(\alpha)^2, \\
i_2(\alpha) &=i_0(\alpha)+|\alpha|^2-2\sigma_2(\alpha)-\sigma_1(\alpha)^2.
\end{align*}
In particular
\begin{equation}\label{eq:i1minusi2}
i_1(\alpha)-i_2(\alpha)=-2i_0(\alpha)+6\sigma_2(\alpha).
\end{equation}
\end{lem}
\begin{proof}
For simplicity put $\alpha_{ij}=\skal{\alpha(e_j)}{e_i}$. Write each $i_1(\alpha)$ and $i_2(\alpha)$ as a sum of two components -- in the first ones (denote them by $\bar{i_1}(\alpha)$ and $\bar{i_2}(\alpha)$, respectively) we take the sum over distinct indices $i,j$. Therefore, the second components are the same for both $i_1(\alpha)$ and $i_2(\alpha)$ and equal
\begin{equation}\label{eq:endcrossprodeq1}
\sum_i |\alpha(e_i)\times e_i|^2=\sum_i |\alpha(e_i)|^2-\sum_i \alpha_{ii}^2=|\alpha|^2-\sum_i \alpha_{ii}^2.
\end{equation}
Moreover, by \eqref{eq:crossprod1} and \eqref{eq:crossprod3} we have
\begin{align*}
\bar{i_1}(\alpha) &=\sum_{i\neq j} \skal{\alpha(e_i)}{e_i\times(\alpha(e_j)\times e_j)}\\
&=-\sum_{i\neq j}\skal{\alpha(e_i)\times\alpha(e)_j}{e_i\times e_j}+\sum_{i\neq j}\alpha_{ii}\alpha_{jj}-2\sum_{i\neq j}\alpha_{ij}\alpha_{ji}\\
&=-i_0(\alpha)+\sum_{i\neq j}\alpha_{ii}\alpha_{jj}-2\sum_{i\neq j}\alpha_{ij}\alpha_{ji}.
\end{align*}
Since $2\sigma_2(\alpha)=\sum_{i\neq j}(\alpha_{ii}\alpha_{jj}-\alpha_{ij}\alpha_{ji})$, together with \eqref{eq:endcrossprodeq1}, we get
\begin{equation*}
i_1(\alpha)=-n(\alpha)+|\alpha|^2+4\sigma_2(\alpha)-\sigma_1(\alpha)^2.
\end{equation*}
Similarly, using \eqref{eq:crossprod1} and \eqref{eq:crossprod3} we have
\begin{align*}
\bar{i_2}(\alpha) &=\sum_{i\neq j}\skal{\alpha(e_i)}{e_j\times(\alpha(e_j)\times e_i)}\\
&=-\sum_{i\neq j}\skal{\alpha(e_i)\times\alpha(e_j)}{e_j\times e_i}+\sum_{i\neq j}\alpha_{ij}\alpha_{ji}-2\sum_{i\neq j}\alpha_{ii}\alpha_{jj}\\
&=i_0(\alpha)+\sum_{i\neq j}\alpha_{ij}\alpha_{ji}-2\sum_{i\neq j}\alpha_{ii}\alpha_{jj}.
\end{align*}
Thus, by \eqref{eq:endcrossprodeq1}, we get
\begin{equation*}
i_2(\alpha)=i_0(\alpha)+|\alpha|^2-2\sigma_2(\alpha)-\sigma_1(\alpha)^2. \qedhere
\end{equation*}
\end{proof}

\begin{rem}\label{rem:alphasym}
Notice that if 
\begin{enumerate}
\item $\alpha=\lambda{\rm Id}$, then
\begin{equation*}
i_0(\alpha)=42\lambda^2,\quad i_1(\alpha)=0,\quad i_2(\alpha)=-42\lambda^2
\end{equation*}
and $\sigma_1(\alpha)=7\lambda$, $\sigma_2(\alpha)=21\lambda^2$, $|\alpha|^2=7\lambda^2$. Thus $i_1(\alpha)-i_2(\alpha)=42\lambda^2$.
\item $\alpha$ is symmetric (not necessary trace--free), then
\begin{equation*}
i_0(\alpha)=2\sigma_2(\alpha),\quad i_1(\alpha)=0,\quad i_2(\alpha)=-2\sigma_2(\alpha).
\end{equation*}
It follows easily by the fact that there is an orthonormal basis such that $\alpha$ is diagonal. In particular,
\begin{equation*}
i_1(\alpha)-i_2(\alpha)=2\sigma_2(\alpha).
\end{equation*}
\item $\alpha\in\mathbb{R}^7$, i.e., $\alpha(u)=u\times Z_{\alpha}$ for some $Z_{\alpha}\in\mathbb{R}^7$, then
\begin{equation*}
i_0(\alpha)=-18|Z_{\alpha}|^2,\quad i_1(\alpha)=36|Z_{\alpha}|^2,\quad i_2(\alpha)=-18|Z_{\alpha}|^2.
\end{equation*}
Moreover, $|\alpha|^2=6|Z_{\alpha}|^2$, $\sigma_2(\alpha)=3|Z_{\alpha}|^2$ and, clearly, $\sigma_1(\alpha)=0$. Therefore,
\begin{equation*}
i_1(\alpha)-i_2(\alpha)=54|Z_{\alpha}|^2.
\end{equation*}
The formulas for $i_0(\alpha),i_1(\alpha)$ and $i_2(\alpha)$ are proven similarly as in the general case in Lemma \ref{lem:endcrossprod} using \eqref{eq:crossprod2}.
\end{enumerate}
\end{rem}

\section{Intrinsic torsion of a $G_2$--structure}

Let $(M,g)$ be an oriented $7$--dimensional Riemannian manifold equipped with a $G_2$--structure. It amounts to the existence of a reduction $P$ of a structure group $SO(7)$ of the orthonormal bundle $SO(M)$ to the subgroup $G_2$. 

We will define a tensor, called the intrinsic torsion, which measures the defect of the Levi--Civita connection $\nabla$ of $g$ to be a $G_2$--connection. Namely, let $\omega$ be a connection form on $SO(M)$ of the Levi--Civita connection $\nabla$. Denote by $\alg{so}(7)$ and $\alg{g}_2$ the Lie algebras of $SO(7)$ and $G_2$, respectively. Then $\omega|TP:TP\to\alg{so}(7)$. We can consider a decomposition
\begin{equation*}
\omega=\omega_{\alg{g}_2}+\omega_{\alg{g}_2^{\bot}},
\end{equation*}
with respect to the orthogonal splitting $\alg{so}(7)=\alg{g}_2\oplus\alg{g}_2^{\bot}$. The component $\omega_{\alg{g}_2^{\bot}}$ is ${\rm ad}(G_2)$--invariant, hence induces a tensor field $-\xi$ on $M$. In fact $\xi\in T^{\ast}M\otimes \alg{g}_2^{\bot}(TM)$, where $\alg{g}_2^{\bot}(TM)$ is a space of endomorphism of the tangent bundles with (skew--)symmetries as in $\alg{g}_2^{\bot}$. It is called the {\it intrinsic torsion} and its study was initiated by Gray and Hervella for $U(n)$--structures \cite{GH}. Moreover,
\begin{equation*}
\xi_XY=\nabla^{G_2}_XY-\nabla_XY,
\end{equation*}
where $\nabla^{G_2}$ is a $G_2$--connection on $M$ induced by the connection form $\omega_{\alg{g}_2}$.

The intrinsic torsion may be described explicitly. Let us briefly show the procedure. Firstly, a $G_2$--structure on $M$ may be equivalently defined as existence of a global $3$--form satisfying locally \eqref{eq:varphi0}. Moreover, there is a cross product $(X,Y)\in TM\times TM\to TM$ defined by
\begin{equation*}
g(X\times Y,Z)=\varphi(X,Y,Z),\quad X,Y,Z\in TM.
\end{equation*}
By the considerations in the previous section, we know that the space $\alg{g}_2^{\bot}(TM)$ consists of the endomorphisms $A_X$ given by $A_X(Y)=Y\times X$. Since $\xi\in T^{\ast}M\otimes \alg{g}_2(TM)\equiv{\rm End}(TM)$ it follows that there is an associated endomorphism $T:TM\to TM$ such that
\begin{equation}\label{eq:inttorT}
\xi_XY=A_{T(X)}Y=Y\times T(X)=-T(X)\times Y,\quad X,Y\in TM.
\end{equation}
Let us compare this formula with the formula by Martin Cabrera \cite{FMC,FMC0}. His approach is the following. The covariant derivative $\nabla\varphi$ satisfies $\nabla\varphi=\xi\varphi$ and the map $\xi\mapsto\varphi$ is $G_2$--equivariant. Thus the space of possible intrinsic torsions may be identified with the space $\mathfrak{X}\subset T^{\ast}M\otimes\Lambda^3(T^{\ast}M)$ of covariant derivatives of $\varphi$. Consider now a $G_2$--equivariant map $r:\mathfrak{X}\to T^{\ast}M\otimes T^{\ast}M$ given by \cite{FMC0}
\begin{equation*}
r(\alpha)(X,Y)=\skal{X\lrcorner\alpha}{Y\lrcorner\star\varphi},\quad X,Y\in TM,
\end{equation*}
where $\skal{\,}{\,}$ is an inner product on the space of $4$--tensors induced from the metric $g$. Then \cite{FMC}
\begin{equation}\label{eq:introrr}
\xi_XY=-\frac{1}{3}\sum_i r(\nabla\varphi)(X,e_i)e_i\times Y,\quad X,Y\in TM.
\end{equation}
Thus, comparing two formulas for the intrinsic torsion \eqref{eq:inttorT} and \eqref{eq:introrr} we have
\begin{equation*}
T(X)=\frac{1}{3}\sum_i r(\nabla\varphi)(X,e_i)e_i
\end{equation*}
or, in other words,
\begin{equation*}
g(T(X),Y)=\frac{1}{3}r(\nabla\varphi)(X,Y).
\end{equation*}
Notice that the assignment $T\mapsto \frac{1}{3}r(\nabla\varphi)$ is just the musical isomorphism, which is $G_2$--invariant.

There is a decomposition into irreducible $G_2$--modules \cite{FG,FMC}
\begin{equation*}
\mathfrak{X}=\mathfrak{X}_1\oplus\mathfrak{X}_2\oplus\mathfrak{X}_3\oplus\mathfrak{X}_4,
\end{equation*} 
which, according to previous considerations and Table 1 in \cite{FMC}, may be characterized with the use of $T$ as follows
\begin{align*}
\textrm{nearly parallel}:\,\mathfrak{X}_1 &: && T=c\,{\rm Id},\quad \textrm{$c$ constant},\\
\textrm{almost parallel}:\,\mathfrak{X}_2 &: && T\in \alg{g}_2(TM),\\
\mathfrak{X}_3 &: && T\in S^2_0(TM),\\
\textrm{locally cofformally parallel}:\,\mathfrak{X}_4 &: && T\in \alg{g}_2^{\bot}(TM).
\end{align*}

\section{An integral formula}

In \cite{KN} the author derived an integral formula for a general Riemannian $G$--structure. The formula involves the intrinsic torsion - its symmetric and skew--symmetric part, the trace and other components related with the curvature. In this section we will state this formula in the case of $G=G_2$.

Let us begin with a general description with the emphasis on the $G_2$ case. For a general approach see, for example, \cite{IA} or \cite{GMC}. Let $\xi$ be the intrinsic torsion of a $G$--structure on an oriented Riemannian manifold $M$. Define
\begin{align*}
\xi^{\rm sym}_XY &=\frac{1}{2}\left(\xi_XY+\xi_YX\right),\\
\xi^{\rm alt}_XY &=\frac{1}{2}\left(\xi_XY-\xi_YX\right),\\
\chi &={\rm tr}\,\xi=\sum_i \xi_{e_i}e_i,
\end{align*} 
where the sum is taken with respect to any orthonormal basis. Since the curvature tensor $R$ is satisfies $R(X,Y)\in\alg{so}(TM)$ we may consider two components $R(X,Y)_{\alg{g}}$ and $R(X,Y)_{\alg{g}^{\bot}}$, where $\alg{g}$ is the Lie algebra of $G$. Thus we may speak about the $\alg{g}^{\bot}$--scalar curvature and we denote it by $s_{\alg{g}^{\bot}}$. Finally, define
\begin{equation*}
s^{\rm alt}_{\alg{g}^{\bot}}=\sum_{i,j}g([\xi_{e_i},\xi_{e_j}]_{\alg{g}^{\bot}}e_j,e_i)
\end{equation*}
with respect to any orthonormal basis $(e_i)$. The main relations concerning these objects are the following.

\begin{thm}\cite{KN}
On an oriented $G$--structure $M$ we have
\begin{equation}\label{eq:divformain}
{\rm div}\chi=\frac{1}{2}s_{\alg{g}^{\bot}}^{\rm alt}-\frac{1}{2}s_{\alg{g}^{\bot}}+|\chi|^2+|\xi^{\rm alt}|^2-|\xi^{\rm sym}|^2.
\end{equation} 
If $M$ is, additionally, closed, then the following integral formula holds
\begin{equation}\label{eq:intformain}
\frac{1}{2}\int_M s_{\alg{g}^{\bot}}-s^{\rm alt}_{\alg{g}^{\bot}}\,{\rm vol}=
\int_M |\chi|^2+|\xi^{\rm alt}|^2-|\xi^{\rm sym}|^2\,{\rm vol}.
\end{equation}
\end{thm}

Specifying to the case of $G_2$ case we have
\begin{align*}
\xi^{\rm sym}_XY &=-\frac{1}{2}(T(X)\times Y-T(Y)\times X),\\
\xi^{\rm alt}_XY &=-\frac{1}{2}(T(X)\times Y+T(Y)\times X),\\
\chi &=-\sum_i T(e_i)\times e_i. 
\end{align*}

\begin{lem}\label{lem:sforg2}
We have
\begin{equation*}
s_{\alg{g}_2^{\bot}}=\frac{1}{3}s\quad\textrm{and}\quad s^{\rm alt}_{\alg{g}_2^{\bot}}=i_0(T),
\end{equation*}
where $s$ denotes the scalar curvature.
\end{lem}
\begin{proof}
Let us first prove the second relation. By \eqref{eq:g2botalternative} and the relation \eqref{eq:crossprod1} for the cross product, we have
\begin{align*}
s^{\rm alt}_{\alg{g}_2^{\bot}} &=\sum_{i,j}g(A_{T(e_i)\times T(e_j)}e_j,e_i)
=\sum_{i,j}g(e_j\times(T(e_i)\times T(e_j)),e_i)\\
&=-\sum_{i,j}g(T(e_i)\times T(e_j),e_j\times e_i)=i_0(T).
\end{align*}

In order to prove the first relation we will use the $\varepsilon$--notation introduced in the previous sections. Notice first than a projection of $R(e_i,e_j)$ to $\alg{g}_2^{\bot}$ is given by $A_X$, where
\begin{equation*}
X=\frac{1}{6}\sum_{k,p,q}\varepsilon_{kpq}R_{ijpq}e_k.
\end{equation*}
Hence
\begin{align*}
s_{\alg{g}_2^{\bot}} &=\frac{1}{6}\sum_{i,j,k,p,q} \varepsilon_{kpq}R_{ijpq}\varepsilon_{jik}
=-\frac{1}{6}\sum_{i,j,p,q}(\varepsilon_{ijpq}+\delta_{ip}\delta_{jq}-\delta_{jp}\delta_{iq})R_{ijpq}\\
&=-\frac{1}{6}\sum_{i,j,p,q}\varepsilon_{ijpq}R_{ijpq}+\frac{1}{3}\sum_{i,j}R_{ijji}.
\end{align*}
By the first Bianchi identity and skew--symmetry of $\varepsilon_{ijpq}$ the first term vanishes. Therefore $s_{\alg{g}_2^{\bot}}=\frac{1}{3}s$.
\end{proof}

We are now ready to state the main result.
\begin{thm}\label{thm:main}
On a closed (i.e., compact without boundary) $G_2$--structure $M$ induced by an endomorphism $T$ the following integral formula holds
\begin{equation}\label{eq:mainint}
\frac{1}{6}\int_M s\,{\rm vol}_M=\int_M -\frac{3}{2}i_0(T)+6\sigma_2(T)\,{\rm vol}_M.
\end{equation}
In particular,
\begin{enumerate}
\item if a $G_2$--structure is of type $\mathfrak{X}_1\oplus\mathfrak{X}_3$, then
\begin{equation}\label{eq:mainint13}
\frac{1}{6}\int_M s\,{\rm vol}_M=3\int_M \sigma_2(T)\,{\rm vol}_M.
\end{equation}
\item if a $G_2$--structure is of pure type $\mathfrak{X}_4$, i.e., $T$ is induced by a vector field $Z$, then
\begin{equation}\label{eq:mainint4}
\frac{1}{3}\int_M s\,{\rm vol}_M=45\int_M |Z|^2 {\rm vol}_M.
\end{equation}
\end{enumerate}
\end{thm}
\begin{proof}
Follows by the fact that
\begin{equation*}
|\chi|^2+|\xi^{\rm alt}|^2-|\xi^{\rm sym}|^2=i_1(T)-i_2(T),
\end{equation*}
by Lemma \ref{lem:sforg2}, Lemma \ref{lem:endcrossprod} and the integral formula \eqref{eq:intformain}. Cases (1), (2) are consequences of Remark \ref{rem:alphasym}.
\end{proof}

\begin{rem} If a $G_2$--structure is of pure type $\mathfrak{X}_1$, i.e., $T=\lambda\,{\rm Id}$ for a constant $\lambda$, then the integral formula becomes trivial with each side equal to $63\lambda^2$. This follows from the fact that a $G_2$--structure of pure type $\mathfrak{X}_1$ is Einstein with the scalar curvature $s=7\cdot 54\lambda^2$. Indeed, by \cite{FKMS} a nearly parallel $G_2$--structure with the defining condition $d\varphi=-8\lambda_0\star\varphi$ is Einstein with the scalar curvature $s=7\cdot 24\lambda_0^2$. On the other hand, by \cite{FMC0} $r=-2\lambda_0\,g$. Therefore, $T=-\frac{2}{3}\lambda_0\,{\rm Id}$, which implies $T=\lambda\,{\rm Id}$, with $\lambda=-\frac{2}{3}\lambda_0$. To compute the right hand side it suffices to use Remark \ref{rem:alphasym}.
\end{rem}

\begin{rem}\label{rem:holonomyG2}
Notice that if the intrinsic torsion vanishes, i.e., the holonomy is in $G_2$, then the integral formula simplifies to $\int_M s\,{\rm vol}_M=0$. This is clear, since integrable $G_2$--structure is necessary Ricci--flat.
\end{rem}

The following fact is a simple implication of the integral formula \eqref{eq:mainint}.

\begin{cor}
If closed (compact, without boundary) $G_2$--structure $M$ satisfies $\int_M s\,{\rm vol}_M=0$, then it cannot be of pure type $\mathfrak{X}_4$ unless it is integrable, i.e., with the holonomy in $G_2$. 
\end{cor}

We wish to notice that in "many cases" the integral formula is in deed a point--wise formula, since the characteristic vector field $\chi$ vanishes. Namely, we have the following result.

\begin{prop}\label{prop:pointwiseformula}
Assume $M$ is a $G_2$--structure of type $\mathfrak{X}_1\oplus\mathfrak{X}_2\oplus\mathfrak{X}_3$ induced by an endomorphism $T$. Then, the following formula holds
\begin{equation*}
\frac{1}{6}s=-\frac{3}{2}i_0(T)+6\sigma_2(T).
\end{equation*}
In particular, the characteristic vector field has only nonzero component in the class $\mathfrak{X}_4$ equal to
\begin{equation*}
\chi_4=-6Z,
\end{equation*}
where $Z$ is a vector field inducing class $\mathfrak{X}_4$.
\end{prop}
\begin{proof}
To prove the first part it suffices to show that if a $G_2$--structure is of type $\mathfrak{X}_1\oplus\mathfrak{X}_2\oplus\mathfrak{X}_3$, then $\chi=0$ (and use \eqref{eq:divformain}). Let $\chi=\chi_1+\chi_2+\chi_3$ is a decomposition of $\chi$ with respect to the decomposition $\xi=\xi_1+\xi_2+\xi_3$. If $T$ is symmetric then we may choose an orthonormal basis such that $T$ has diagonal form. Thus $T(e_i)\times e_i=0$ for any index $i$, so $\chi_1+\chi_3=0$. Moreover, if $T\in\alg{g}_2(TM)$, then with the $\varepsilon$--notation
\begin{equation*}
\sum_i T(e_i)\times e_i=\sum_{i,j}t_{ij}e_j\times e_i
=\sum_{i,j,k}t_{ij}\varepsilon_{jik}e_k
=-\sum_k(\sum_{i,j}\varepsilon_{kij}t_{ij})e_k=0,
\end{equation*}
by the definition of $\alg{g}_2$. 

For the proof of the second part, if $T(X)=X\times Z$, then by \eqref{eq:crossprod2},
\begin{equation*}
\chi_4=\sum_i e_i\times T(e_i)=\sum_i e_i\times (e_i\times Z)=-6Z.
\end{equation*}
\end{proof}

\subsection{Correlation with the formula by Bor and Hernandez Lamoneda}

In \cite{BH1} authors apply Bochner type formula
\begin{equation*}
\int_M \skal{\tilde{R}\omega}{\omega}\,{\rm vol}_M=\int_M |d\omega|^2+|d^{\star}\omega|^2-p!|\nabla\omega|^2\,{\rm vol}_M
\end{equation*}
for a compact $M$, where $\omega$ is a $p$--form and $\tilde{R}$ is a curvature tensor acting on differential forms. Choosing $\omega$ as a $3$--form defining a $G_2$--structure on a $7$--dimensional Riemannian manifold $(M,g)$, authors in \cite{BH1} obtain the following integral formula (assuming $M$ is compact)
\begin{equation*}
\frac{2}{3}\int_M s\,{\rm vol}_M=\int_M 6|\tau_1|^2+5|\tau_7|^2-|\tau_{14}|^2-|\tau_{27}|^2\,{\rm vol}_M.
\end{equation*}
Here, $s$ is the scalar curvature. Let us give the definitions of $\tau_i$. By the representation theory arguments there is the following decomposition into irreducible $G_2$--modules
\begin{equation*}
W=\Lambda^1(T^{\ast}M)\otimes(\alg{g}_2^{\bot}\cdot\omega)=W_1\oplus W_7\oplus W_{14}\oplus W_{27},
\end{equation*}
where the index denotes the dimension of the space. Since $\nabla\omega\in W$, we denote its component $(\nabla\omega)_i$ in $W_i$ by $\tau_i$. Hence, comparing with our notation, $\tau_1$ corresponds to $\mathfrak{X}_1$, $\tau_7$ with $\mathfrak{X}_4$, $\tau_{14}$ with $\mathfrak{X}_2$ and $\tau_{27}$ with $\mathfrak{X}_3$.

\subsection{Correlation with the formula by Bryant}

We have already mentioned the correlation of the main integral formula \eqref{eq:mainint}. Let us indicate how the main formula \eqref{eq:mainint}, or even the divergence formula \eqref{eq:divformain} adapted to the setting of $G_2$--structures (compare Proposition \ref{prop:pointwiseformula}), is related with the results of R. Bryant in \cite{RB}. In order to state the divergence formula by Bryant we need some preparations. We define {\it torsion forms} $\tau_1,\ldots, \tau_4$, corresponding to intrinsic torsion modules, as follows
\begin{align*}
d\varphi &=\tau_0\star_{\varphi}\varphi+3\tau_1\wedge\varphi+\star_{\varphi}\tau_3, \\
d\star_{\varphi}\varphi &=4\tau_1\wedge\star_{\varphi}\varphi+\tau_2\wedge\varphi,
\end{align*} 
where $\varphi$ is a $3$--form defining a $G_2$--structure. Here, $\tau_0$ is a function ($0$--form), $\tau_1$ a $1$--form, $\tau_2\in \Lambda^2_{(14)}(T^{\ast}M)$ and $\tau_3\in\Lambda^3_{(27)}(T^{\ast}M)$. With respect to our notation: $\tau_0$ corresponds to $\mathfrak{X}_1$, $\tau_1$ with $\mathfrak{X}_4$, $\tau_2$ with $\mathfrak{X}_2$ and $\tau_3$ with $\mathfrak{X}_3$.

It appears that the scalar curvature $s$ can be described with the use of torsion forms. The formula by Bryant \cite{RB} is the following
\begin{equation}\label{eq:Bryant}
s=12\delta\tau_1+\frac{21}{8}\tau_0^2+30|\tau_1|^2-\frac{1}{2}|\tau_2|^2-\frac{1}{2}|\tau_3|^2,
\end{equation}
where $\delta$ is a codifferential (divergence). Above formula may be treated as a divergence formula \eqref{eq:divformain} combined with the Proposition \ref{prop:pointwiseformula}. Integrating \eqref{eq:Bryant}, with the assumption that a $G_2$--structure $M$ is a closed manifold, we get an analogue of the integral formula \eqref{eq:mainint}:
\begin{equation*}
\int_M s\,{\rm vol}_M=\int_M \frac{21}{8}\tau_0^2+30|\tau_1|^2-\frac{1}{2}|\tau_2|^2-\frac{1}{2}|\tau_3|^2\,{\rm vol}_M.
\end{equation*}

\subsection{Correlation with the formula by Friedrich and Ivanov}

In \cite{FI} the authors considered three types of $G$--structures with totally skew-symmetric torsion of the $G$--connection (compare also \cite{AF} and \cite{AF2} on such types of $G$--structures). They showed that for $G=G_2$ such assumption is satisfied if and only if a $G_2$--structure is of type $\mathfrak{X}_1\oplus\mathfrak{X}_3\oplus\mathfrak{X}_4$. Denoting by $\varphi$ the defining $3$--form for $G_2$--structure the torsion ${\rm Tor}$ is of the form \cite{FI} (with respect to our notation)
\begin{equation}\label{eq:torsion}
{\rm Tor}=\frac{1}{6}(d\varphi,\star\varphi)\varphi-\star d\varphi+\star(Z^{\flat}\wedge\varphi),
\end{equation}   
where $Z$ is a vector field inducing $\mathfrak{X}_4$--component of $G_2$--structure. Moreover, in \cite{FI2} these authors computed the scalar curvature with the use of the given data as follows
\begin{equation}\label{eq:scalarskewtorsion}
s=\frac{1}{18}(d\varphi,\star\varphi)^2+2|Z|^2-\frac{1}{12}|{\rm Tor}|^2+3{\rm div}(Z).
\end{equation}
In particular \cite{AF2}, if the $G_2$--structure is of type
\begin{itemize}
\item $\mathfrak{X}_1$, then $s=\frac{27}{2}|{\rm Tor}|^2$,
\item $\mathfrak{X}_3$, then $s=-\frac{1}{2}|{\rm Tor}|^2$.
\end{itemize}
Notice, that by Proposition \ref{prop:pointwiseformula} and Remark \ref{rem:alphasym} we have
\begin{itemize}
\item for type $\mathfrak{X}_1$: $\frac{1}{6}s=63\lambda^2$, where $\lambda$ is a constant defining a structure, i.e., $T=\lambda\,{\rm Id}$,
\item for type $\mathfrak{X}_3$: $\frac{1}{6}s=3\sigma_2(T)$, where $T$ is a symmetric trace-free endomorphism defining a structure. Hence, $\sigma_2(T)$ is non--positive.
\end{itemize}

\section{Examples}

Let us verify the main integral formula on appropriate examples.

\begin{exa}
Let $M$ be an oriented hypersurface in $\mathbb{R}^8$. It is known \cite{FG} that there is a cross product in the tangent bundle of $M$. Namely, let $P:(\mathbb{R}^8)^{\times 3}\to\mathbb{R}^8$ be a $3$--fold cross product which exists by the result of Brown, Gray \cite{BG} (there are two non--isomorphic such cross products - we choose one of them). By orientability of $M$ there is a global unit normal vector field $N$ to $M$. It suffices to put
\begin{equation*}
X\times Y=P(N,X,Y),\quad X,Y\in T_xM,\quad x\in M.
\end{equation*}
The associated $3$--form $\varphi(X,Y,Z)=\skal{X\times Y}{Z}$ defines a $G_2$--structure on $M$. The following result holds \cite{FG}.
\begin{thm}[\cite{FG}]\label{thm:hypersurfaces}
A $G_2$--structure on $M$ is of type $\mathfrak{X}_1\oplus\mathfrak{X}_3$. Moreover, it is of pure type
\begin{enumerate}
\item $\mathfrak{X}_1$ if and only if $M$ is totally umbilical,
\item $\mathfrak{X}_3$ if and only if $M$ is minimal.
\end{enumerate}
In particular, a $G_2$--structure is parallel, i.e. the intrinsic torsion vanishes, if and only if $M$ is totally geodesic.
\end{thm}

Let us now rewrite the main integral formula. Denote by $S:TM\to TM$ the shape operator $S(X)=-\nabla_XN$. $S$ is symmetric and dual to the second fundamental form $B:TM\times TM\to\mathbb{R}$, $B(X,Y)=g(S(X),Y)$, $X,Y\in TM$. It can be shown that \cite{FG}
\begin{equation*}
(\nabla_X\varphi)(Y,Z,W)=\pm\star\varphi(S(X),Y,Z,W).
\end{equation*}
depending on the chosen cross product. Thus
\begin{align*}
r(\nabla\varphi)(X,Y) &=\frac{1}{4}\skal{S(X)\lrcorner\nabla\varphi}{Y\lrcorner(\ast\varphi)}=
\pm\frac{1}{4}\sum_{i,j,k}\star\varphi(S(X),e_i,e_j,e_k)\star\varphi(Y,e_i,e_j,e_k)\\
&=\pm\frac{1}{4}\sum_{i,j,k}\skal{S(X)}{(e_i\times e_j)\times e_k}\skal{Y}{(e_i\times e_j)\times e_k}\\
&=\pm 8g(S(X),Y)=\pm 8B(X,Y).
\end{align*}
Therefore
\begin{equation*}
T(X)=\pm \frac{8}{3}S(X),\quad X\in TM.
\end{equation*}
Clearly, $T$ is symmetric. Moreover, $T={\rm const}\cdot{\rm Id}$ implies $S$ is proportional to the identity, whereas, $T\in S^2_0(TM)$ implies ${\rm tr}(S)=0$. This justifies Theorem \ref{thm:hypersurfaces}. 

An integral formula \eqref{eq:mainint13} takes the form
\begin{equation*}
\int_M s\,{\rm vol}_M=128\int_M \sigma_2(S)\,{\rm vol}_M.
\end{equation*}
The function $\sigma_2(S)$ is called the $2$nd mean curvature of a hypersurface. Thus we have a relation between two different types of curvatures -- the former $s$ belongs to the intrinsic geometry, the latter $\sigma_2(S)$ to the extrinsic geometry of a hypersurface. Notice, that by Proposition \ref{prop:pointwiseformula}, the above integral formula is in fact, pointwise formula
\begin{equation*}
s=128\sigma_2(S).
\end{equation*}
\end{exa}

\begin{exa}
Consider a product space $M=(H(1,2)/\Gamma)\times\mathbb{T}^2$, where $\mathbb{T}^2$ is the $2$--dimensional torus, $H(1,2)$ is a generalized Heisenberg group consisting of matrices of the form
\begin{equation*}
\left(\begin{array}{cccc} 1 & 0 & x_1 & z_1 \\ 0 & 1 & x_2 & z_2 \\ 0 & 0 & 1 & y \\ 0 & 0 & 0 & 1 \end{array}\right),\quad x_1,x_2,z_1,z_2,y\in\mathbb{R},
\end{equation*}
and $\Gamma$ is a discrete subgroup of matrices with integer coefficients \cite{FMC}. Let us define a basis of one--forms. Forms $dx_1, dx_2, dy, dz_1-x_1dy, dz_2-x_2dy$ on $H(1,2)$ induce forms $\beta_1,\beta_2,\lambda,\gamma_1,\gamma_2$ on $H(1,2)/\Gamma$. Following \cite{FMC} we put
\begin{equation*}
e^0=\beta_1,\, e^1=\gamma_2,\, e^2=\eta_2,\, e^3=\eta_1,\, e^4=\beta_2,\, e^5=\lambda,\, e^6=\gamma_1,
\end{equation*} 
where $\eta_1$ and $\eta_2$ are one--forms forming a basis for the torus $\mathbb{T}^2$ (the use indices starting from $0$ not from $1$ will be justified below by the convention for the cross product). It turns out that $M$ is a $G_2$--structure of pure type $\mathfrak{X}_2$. Moreover, $r(\varphi)=\frac{1}{2}(e^{01}-e^{46})$, which implies the formula for $T$ with respect to considered (dual) basis 
\begin{align*}
& T(e_0)=\frac{1}{6}e_1 && T(e_1)=-\frac{1}{6}e_0, && T(e_2)=0, && T(e_3)=0,\\
& T(e_4)=-\frac{1}{6}e_6, && T(e_5)=0, && T(e_6)=\frac{1}{6}e_4.
\end{align*} 
Therefore, the characteristic polynomial of $T$ equals
\begin{equation*}
\chi_T(t)=\det(T-t I)=-t^7-\frac{1}{18}t^5-\frac{1}{6^4}t^3.
\end{equation*}
It implies that $\sigma_2(T)=\frac{1}{18}$. To derive the formula for $i_0(T)$ we need the formula for the cross product. It satisfies 
\begin{equation*}
e_i\times e_{i+1}=e_{i+3},\quad e_i\times e_{i+3}=-e_{i+1},\quad e_{i+1}\times e_{i+3}=e_i,\quad i\in\mathbb{Z}_7.
\end{equation*}
In this case the frame $(e_0,e_1,\ldots,e_6)$ is called Cayley \cite{FMC}. After some computations we derive $i_0(T)=\frac{1}{3}$. Thus, the component under the integral sign on the right hand side of \eqref{eq:mainint} equals
\begin{equation}\label{eq:ex1RHS}
-\frac{3}{2}i_0(T)+6\sigma_2(T)=-\frac{1}{2}+\frac{1}{3}=-\frac{1}{6}.
\end{equation}

Now let us compute the scalar curvature $s$. It can be shown \cite{FMC} that the Levi-Civita connection satisfies
\begin{align*}
& \nabla_{e_0}e_5=\frac{1}{2}e_6, && \nabla_{e_0}e_6=-\frac{1}{2}e_5, && \nabla_{e_1}e_4=-\frac{1}{2}e_5, && \nabla_{e_1}e_5=\frac{1}{2}e_4, \\
& \nabla_{e_4}e_1=-\frac{1}{2}e_5, && \nabla_{e_4}e_5=-\frac{1}{2}e_1, && \nabla_{e_5}e_0=-\frac{1}{2}e_6, && \nabla_{e_5}e_1=\frac{1}{2}e_4,\\
& \nabla_{e_5}e_4=-\frac{1}{2}e_1, && \nabla_{e_5}e_6=\frac{1}{2}e_0, && \nabla_{e_6}e_0=-\frac{1}{2}e_5, && \nabla_{e_6}e_5=\frac{1}{2}e_0 
\end{align*}
with the remaining elements equal to zero. Thus nonzero coefficients $R_{ijji}$ are only for the following $12$ pairs $(i,j)$:
\begin{equation*}
(5,0), (6,0), (4,1), (5,1), (1,4), (5,4), (0,5), (1,5), (4,5), (6,5), (0,6), (5,6).
\end{equation*}
Two components are equal $-\frac{1}{4}$, three are equal $-\frac{3}{4}$, whereas the remaining seven are equal $\frac{1}{4}$. Finally, $s=-1$. Comparing with \eqref{eq:ex1RHS} we see that the integral formula \eqref{eq:mainint} is satisfied.  
\end{exa}

\end{document}